\documentclass[5pt]{article}
\usepackage{amsmath}
\usepackage[latin1]{inputenc}
\usepackage{amsmath,amsthm,amssymb}
\usepackage{graphicx}
\usepackage{color}

\renewcommand{\theequation}{\thesection.\arabic{equation}}
\newtheorem{defn}{Definition}[section]
\newtheorem{lem}{Lemma}[section]
\newtheorem{thm}{Theorem} [section]
\newtheorem{prop}{Proposition} [section]

\newtheorem{coro}{Corollary}[section]

\title{Nonsingular splittings over finite fields
\thanks{Supported By NSF of China No. 11671153 }}

\author{Pingzhi Yuan\thanks{ P. Yuan is with School of  of Mathematical Science, South China Normal University,  Guangzhou 510631, China (email: yuanpz@scnu.edu.cn).}}

    \date{}
\begin{document}
\baselineskip15pt \maketitle
\renewcommand{\theequation}{\arabic{section}.\arabic{equation}}
\catcode`@=11 \@addtoreset{equation}{section} \catcode`@=12

    \begin{abstract} Let $G$ be a finite abelian group.
  We  say that $M$ and $S$ form a \textsl{splitting} of $G$ if every nonzero element $g$ of $G$ has a unique representation of the form $g=ms$ with $m\in M$ and $s\in S$, while $0$ has no such representation.
The splitting is called {\it nonsingular} if $\gcd(|G|, a) = 1$ for any $a\in M$.

  In this paper, we focus our study on  nonsingular splittings of cyclic groups.  We introduce a new notation --direct KM logarithm and we prove that if there is a prime $q$ such that $M$ splits $\mathbb{Z}_q$, then there are infinitely many primes $p$ such  that $M$ splits $\mathbb{Z}_p$.

\end{abstract}


\section{Introduction}

Splittings were first considered in \cite{St67} in connection with the problem of tiling the Euclidean space by translates of certain polytopes composed of unit cubes, called $k$-crosses and $k$-semicrosses, see also \cite{HS86} and \cite{ St84, SS94, Sz86, Sz87, SzS09}. Splitter sets are equivalent to codes correcting
single limited magnitude errors in flash memories (see \cite{BE13}, \cite{EB10, HP03, KBE11, KLNY11, KLY12, M96, OSW18, Sc12, Sc14, YKB13} and the references
therein).

Let $G$ be a finite group, written additively, $M$ a set of integers, and $S$ a subset of $G$.
We  say that $M$ and $S$ form a \textsl{splitting} of $G$ if every nonzero element $g$ of $G$ has a unique representation of the form $g=ms$ with $m\in M$ and $s\in S$, while $0$ has no such representation.
(Here $ms$ denotes the sum of $m$ $s$'s if $m\geq 0$, and $-((-m)s)$ if $m<0$.)
We  write $G\setminus \{0\}=MS$ to indicate that $M$ and $S$ form a splitting of $G$.
$M$  is referred to as the multiplier set and $S$ as the splitter set.
We  also say that $M$ splits $G$ with  a splitter set $S$, or simply that $M$ splits $G$. A splitting $G\setminus \{0\}=MS$ of a finite group $G$ is called {\it nonsingular} if every element of $M$ is relatively prime to $|G|$, otherwise the splitting is called {\it singular}.

The following notations are fixed throughout this paper.

$\bullet$  For an odd prime $p$, a primitive root $g$ modulo $p$, and an
integer $b$ not divisible by $p$, there exists a unique integer
$l\in [0, p - 2]$ such that $g^l\equiv b \pmod{ p}$.  It is known as
the index of $b$ relative to the base $g$, and it is denoted by
$ind_g(b)$.

$\bullet$ For any positive integer $q$, let $\mathbb{Z}_q$ be the ring of integers
modulo $q$ and $\mathbb{Z}_q ^\ast= \mathbb{Z}_q \backslash \{0\}$. For $a\in\mathbb{Z}_q ^\ast$, $o(a)$ denotes the order of $a$ in the multiplicative group $\mathbb{Z}_q ^\ast$.

$\bullet$  Let $a, b$ be integers such that $a\le b$, denote
$$[a, b] = \{a, a + 1, a + 2, \ldots , b\}\,\, \mbox{and}$$
$$[a, b]^\ast = \{a, a + 1, a + 2, \ldots, b\}\backslash\{0\}.$$

$\bullet$ Unless additionally defined, we assume that $aT=a\cdot T = \{a\cdot t : t\in T \}$, \,  $A+B=\{a+b, a\in A, b\in B\}$ and $AB=A\cdot B = \{a\cdot b, a\in A, b\in B\}$ for any element
$a$ and any sets $A$ and $B$, where $\cdot$ and $+$ are binary operators.

$\bullet$ For a nonempty set $M$, $|M|$ denotes the number of elements in $M$.
\begin{defn}Let $(G, \cdot)$ be an abelian group (written multiplicatively). If each element $g\in G$ can be expressed uniquely in the form
$$g = a \cdot b, a \in A,\,\, b \in B,$$
then the equation $G = A \cdot B$ is called a {\it factorization} of $G$. A non-empty subset of $G$ is called to be a {\it direct  factor} of $G$ if there exists a subset $B$ such that $G=A\cdot B$ is a factorization. \end{defn}

In 1983, Hickerson \cite{H83} proved the following result.
\begin{prop}{\rm (\cite{H83}, Theorem 2.2.3)}\label{nonsingular}
Let $G$ be a finite group and $M$ a set of nonzero integers.
Then $M$ splits $G$ nonsingularly if and only if $M$ splits $\mathbb{Z}_p$ for each prime divisor $p$ of $|G|$.
\end{prop}

By the above proposition of Hickerson \cite{H83}, for the study of nonsingular splittings of abelian groups we can restrict the study to the cyclic group  $\mathbb{Z}_p$. In this paper, we focus our study on  nonsingular splittings of $\mathbb{Z}_p$ for some prime $p$.

The arrangement of the paper is as follows: In Section 2, we introduce a new notation. By using a powerful of Kummer and Mills, we prove the main results of this paper in Section 2. In Section 3, we present the related results when $M=[-k_1, k_2]^*$. In Section 4, we give a characterization of  the possible splitter sets $B$ such that $[-1, 5]^*$ splits $\mathbb{Z}_p$
with the splitter set $B$.

\section{Direct KM logarithm and nonsingular splittings}

We apply a powerful theorem first proved by Kummer and
generalized by Mills \cite{[M]} to the splittings of cyclic groups.

A $k$-character $\chi$ on $\mathbb{Z}_p$ ($p$ is a prime) is a homomorphism from $\mathbb{Z}_p^\star$ to
$\mathbb{Z}_k$. Let $p_1, \ldots, p_t$ be distinct primes and let $b_1, \ldots, b_t$ be elements of $\mathbb{Z}_k$.
Mills \cite{[M]} obtained the following necessary and sufficient conditions for the existence
of a prime $p$ and corresponding $k$-character $\chi$ such that $\chi(p_i)= b_i, 1\le  i\le t$.

\begin{thm}\label{th21} (Kummer-Mills) Let $p_1, \ldots, p_t$  be distinct primes and let
$b_1, \ldots, b_t\in \mathbb{Z}_k$. There is an infinite number of primes $p$ and $k$-characters
$\chi : \left(\mathbb{Z}_p\right)^\ast\to \mathbb{Z}_k$ such that $\chi(p_i) = b_i$ for $1 \le i \le t$ if and only if

(1) $k$ is odd,

(2) $k = 2m$ where $m$ is odd and (a) for each $p$ such that $p\equiv1\pmod{4}$ and $p$
divides $m$, $b_i$ is even and (b) for all $p\equiv 3\pmod{4}$ which divide $m$, the corresponding $b_i$ all have the same parity,

(3) $k = 4m$ and for each $p_i$ that divides $m$, $b_i$ is even.
Moreover, if there is one such prime $p$ for which a $k$-character exists with prescribed
values at $p_1, \ldots, p_t$, then there is an infinite number. \end{thm}

If $p_1=-1$, then we have the following result.

\begin{thm}\label{th22} (Kummer-Mills) Let $p_1=-1$,  $p_2 \ldots, p_t$  be distinct primes and let
$b_1, \ldots, b_t\in \mathbb{Z}_k$. There is an infinite number of primes $p$ and $k$-characters
$\chi : \left(\mathbb{Z}_p\right)^\ast\to \mathbb{Z}_k$ such that $\chi(p_i) = b_i$ for $1 \le i \le t$ if and only if

(1) $k = 2m$ where $m$ is odd and (a) for each $p$ such that $p\equiv1\pmod{4}$ and $p$
divides $m$, $b_i$ is even,  and (b)$b_1\equiv k/2\pmod{k}$ is odd and  for all $p\equiv 3\pmod{4}$ which divide $m$,  all the corresponding $b_i$are odd,

(2) $k = 4m$ where $m$ is odd and and (a) for each odd prime $p_i$ that divides $m$, $b_i$ is even and (b) $b_1\equiv k/2\pmod{k}$ and $b_i$ is odd for $p=2$.

(3) $k = 8m$ and for each prime $p_i$ that divides $m$, $b_i$ is even.
Moreover, if there is one such prime $p$ for which a $k$-character exists with prescribed
values at $p_1, \ldots, p_t$, then there is an infinite number. \end{thm}

Let $k_1$ and $k_2$ be non-negative integers with $k_1\le k_2$ and $k_1+k_2=k\ge3$, a {\it logarithm \, function} (of length $k$) is a function  $f: [-k_1, k_2]^\ast\to \mathbb{Z}_k$ such that
$f(xy) = f(x)+ f(y)$ whenever $ x, y, xy \in[-k_1, k_2]^\ast$. A {\it logarithm} is a bijective logarithmic function. Logarithms are used in lattice tilings, group theory, number theory, coding theory, and $k$-radius sequences, see  Blackburn and Mckee \cite{BM12} and the references therein.

Let $p$ be a prime such that $p\equiv1\pmod{k}$, a $k$-character $\chi$ on $\mathbb{Z}_p$ ($p$ is a prime) is a homomorphism from $\mathbb{Z}_p^\star$ to
$\mathbb{Z}_k$. If the $k$-character $\chi$ is primitive, i.e., the homomorphism is an epimorphism, then there is a primitive $k$th root of unity $\zeta$  in $\mathbb{Z}_p$ such that $\chi$ is determined by the equation $x^{(p-1)/k}\equiv \zeta^{\chi(x)}\pmod{p}$. Observe that the above determined function $\chi$ satisfies $\chi(ab)\equiv \chi(a)+\chi(b)\pmod{k}$ for any $a, b\in\mathbb{Z}_p$ and the image of $\chi$ is $\mathbb{Z}_k$.  We say (following Galovich and
Stein [6], Blackburn and Mckee \cite{BM12}) that $f$ is a Kummer-Mills-logarithm, or KM-logarithm, if it arises in this
way for some prime $p$.

Now we introduce the following general definitions.
\begin{defn}Let $M$ be a finite subset of nonzero integers.  An injective function $f: M\to \mathbb{Z}_k$ is called a {\it direct logarithm} if
$f(xy) = f(x)+ f(y)$ when $ x, y, xy \in M$ and $f(M)$ is a direct factor of $\mathbb{Z}_k$. A {\it direct logarithm} that meets the
conditions of Theorems \ref{th21} and \ref{th22} is called a {\it direct KM-logarithm}.\end{defn}

It is easy to see that $|M||k$ and the direct KM-logarithm is the usual   KM-logarithm when $|M|=k$. Hence it is a generalization of the usual   KM-logarithm. We first prove the following useful proposition.

\begin{prop}\label{mainprop} Let $M$ be a finite set of nonzero integers. Suppose that there is a direct  logarithm $f: M\to \mathbb{Z}_k$ for some positive integer $k$ with $|M||k$, then the function $g:
M\to \mathbb{Z}_{8k}$ defined by $g(m)\equiv8f(m)\pmod{8k}$ is a direct KM logarithm. \end{prop}
\begin{proof} Observe that $g(m)$ is even for any $m\in M$, so it suffices to show that $g$ is a direct logarithm. Since $f$ is a direct logarithm, so $f$ is injective and there is a subset of $\mathbb{Z}_k$ such that $f(M)+B=\mathbb{Z}_k$ is a factorization. Let $B_1=\{8b\pmod{8k}, b\in B\}$. We will show that
$$g(M)+B_1+\{0, 1, 2, 3, 4, 5, 6, 7\}=\mathbb{Z}_{8k}$$
is a factorization. For any element $a\in \mathbb{Z}_{8k}$, $0\le a<8k$, it is easy to see that
$$a=a_1+8t, \quad a_1\in \{0, 1, 2, 3, 4, 5, 6, 7\}, \,\, 0\le t<k$$ and the representation is unique. Recall that $f(M)+B=\mathbb{Z}_k$ is a factorization, so we have $t=f(m)+b, b\in B, m\in M$ and the representation is unique. Hence
$$a=a_1+8f(m)+8b=a_1+g(m)+b_1,\quad a_1\in \{0, 1, 2, 3, 4, 5, 6, 7\},\quad  b_1\in B_1,$$ and the representation is unique, as required.\end{proof}

\begin{thm}\label{main} Let $M$ be a finite set of nonzero integers. Suppose that there is a direct KM logarithm $f: M\to \mathbb{Z}_k$ for some positive integer $k$ with $|M||k$, then there are infinitely
many primes $p$ such  that $M$ splits $\mathbb{Z}_p$. \end{thm}
\begin{proof} By the definition of a direct KM logarithm and Kummer-Mills Theorem, there are infinitely
many primes $p$ such  that the restriction of the $k$-th power character map $\varphi: \mathbb{Z}_p^*\to \mathbb{Z}_p^*$ defined by $\varphi(a)\equiv a^{\frac{p-1}{k}t}\pmod{p}, \gcd(t, k)=1$ to $M$ is injective and $\varphi(M)$ is a direct factor of the cyclic group $\varphi(\mathbb{Z}_p^*)$ of order $k$. We will show that for all these primes $p$, $M$ splits $\mathbb{Z}_p$.

Since $\varphi(M)$ is a direct factor of the cyclic group $\varphi(\mathbb{Z}_p^*)$, so there is  a finite subset $B$ of $\varphi(\mathbb{Z}_p^*)$ such that  $1\in B$ and  $\varphi(M)B=\varphi(\mathbb{Z}_p^*)$ is a factorization. Let $B_1=\{a\in\mathbb{Z}_p^*, \varphi(a)\in B\}$. Then $|B_1|=|B|\cdot\frac{p-1}{k}$ and $p-1=|M||B_1|$. Therefore it suffices to show that $MB_1$ is direct. Suppose now that $m_1b_1=m_2b_2, m_1, m_2\in M$ and $b_1, b_2\in B_1$, then
$$ \varphi(m_1)\varphi(b_1)=\varphi(m_2)\varphi(b_2).$$
Notice that $\varphi(m_1), \varphi(m_2)\in \varphi(M)$ and $\varphi(b_1), \varphi(b_2)\in \varphi(B_1)=B$ and $\varphi(M)B=\varphi(\mathbb{Z}_p^*)$ is a factorization. Hence $\varphi(m_1)=\varphi(m_2)$, which implies that $m_1=m_2$ since the restriction of $\varphi$ to $M$ is injective, and thus $b_1=b_2$. This completes the proof. \end{proof}

\begin{thm} Let $M$ be a set of integers. Suppose that there is a prime $q$ such that $M$ splits $\mathbb{Z}_q$, then there are infinitely
many primes $p$ such  that $M$ splits $\mathbb{Z}_p$. \end{thm}
 \begin{proof}Since $q$ is a prime  and $M$ splits $\mathbb{Z}_q$, so restriction of the $q-1$-th power character map $\varphi: \mathbb{Z}_q^*\to \mathbb{Z}_q^*$ defined by $\varphi(a)\equiv a\pmod{p}$ to $M$ is injective and $\varphi(M)$ is a direct factor of the cyclic group $\varphi(\mathbb{Z}_q^*)$ of order $q-1$. It is a direct  logarithm, therefore the result follows from Proposition \ref{mainprop} and Theorem \ref{main}.\end{proof}

To illustrate, let us give two examples.

{\bf Example 2.1:}  For the set $[-4, 4]^*$, we first show that there is no KM logarithm $f:[-4, 4]^*\to \mathbb{Z}_8$. The reason is that: for any logarithm $f:[-4, 4]^*\to \mathbb{Z}_8$, we have $f(-1)=4$, $f(1)=0$. If $f$ were a KM logarithm, then $f(2)$ should be even. It follows that $f(2), f(-2), f(4), f(-4)$ are even, so $f(2), f(-2), f(4), f(-4), f(-1), f(1)$ are even, which contradicts to that $f$ is a logarithm. However, we have the following direct KM logarithm $g:[-4, 4]^*\to \mathbb{Z}_{16}$ given by $g(1)=0$, $g(-1)=8$, $g(2)=2$, $g(4)=4$, $g(-2)=10$, $g(-4)=12$, $g(3)=6$, $g(-3)=14$. We have $g([-4, 4]^*)\bigoplus\{-1, 1\}=\mathbb{Z}_{16}$ and $g(2)$ is even, so $g$ is a direct KM logarithm, which implies that there are infinitely many primes $p$ such that $[-4, 4]^*$ splits $\mathbb{Z}_p$ by Theorem \ref{main}.

{\bf Example 2.2:}  For other cases with $k_1+k_2=8$, we have
$$(-1, 1, 2, 3, 4, 5, 6, 7)\mapsto(4, 0, 1, 6, 2, 3, 7, 5); (-2, -1, 1, 2, 3, 4, 5, 6)\mapsto(5, 4, 0, 1, 6, 2, 3, 7);$$
$$(-3, -2, -1, 1, 2, 3, 4, 5)\mapsto(7, 5, 4, 0, 1, 3, 2, 6)$$
are logarithms from $[-k_1, k_2]^*$ to $\mathbb{Z}_8$. It is easy to see that none of  them is a KM logarithm. By the same argument as in Example 1, there are direct KM logarithms from $[-k_1, k_2]^*$ to $\mathbb{Z}_{16}$. Therefore there are infinitely many primes $p$ such that $[-k_1, k_2]^*$ ($[-k_1, k_2]^*=[-1, 7]^*$ or $[-2, 6]^*$ or $[-3, 5]^*$) splits $\mathbb{Z}_p$.

We say that a prime $p$ is a $k$-radius prime if the following two conditions both hold:

(a) $p\equiv1\pmod{2k}$;

(b) the elements $1^{(p-1)/k}, 2^{(p-1)/k}, \ldots, k^{(p-1)/k}$ in  $\mathbb{Z}_p^*$ are pairwise distinct.

In \cite{BM12},  Blackburn and Mckee  proved that  there is a special KM-logarithm of length $k$ if and only if there are infinitely many $k$-radius primes. Let $k$ be a fixed positive integer, and let $f_{spec}(k)$ be the number of special KM-logarithms of length $k$, then there exists two positive constant $d_k$ and $A_k$ such that the number of $k$-radius primes less
than or equal to $x$ is
$$d_kf_{spec}(k)\frac{x}{\log x} + O(x \exp(-A_k\log x),$$
as $x\to\infty$, where the implied constants depends only on $k$ by Theorem 1 in Elliott \cite{E70}.

 We have

\begin{prop}Let $k$ be a positive integer. Suppose that the prime $p$ is a $k$-radius prime, then both $[1, k]$ and $[-k, k]^\star$ split $\mathbb{Z}_p$.\end{prop}
\begin{proof}Let  $\varphi:\mathbb{Z}_p^*\to\mathbb{Z}_p^*, \varphi(a)\equiv a^{(p-1)/k}$ be the power residue map. Then we have $\varphi(ab)=\varphi(a)\varphi(b), a, b\in\mathbb{Z}_p^*$. It is easy to check that $[1, k]ker(\varphi)=\mathbb{Z}_p^*$ since $p$ is a $k$-radius prime. Note that $\{-1,1\}$ is a subgroup of $ker(\varphi)$, so $ker(\varphi)=\{-1, 1\}B$ is a factorization for some $B\in\mathbb{Z}_p^*$, which implies $[-k, k]^*B=\mathbb{Z}_p^*$ is also a factorization. This completes the proof.\end{proof}

\section{Existence and nonexistence of split   $[-k_1, k_2]^*$ sets}

The cases $M=[1, k]$ and $[-k, k]^*$ arise in the case of tiling Euclidean space by certain star bodies (see \cite{St67}). Motivated by an application to error-correcting codes for non-volatile memories, Schwarz suggested in \cite{Sc14} to consider the cases $[-k_1, k_2]^*, 1\le k_1\le k_2$. In  \cite{H73}, Hamarker proved that $M=\{1, 3, 27\}$ splits no finite abelian group.

In this section, we consider the case when $M=[-k_1, k_2]^*$, where $k_1\le k_2$ are non-negative integers.

 By Proposition \ref{mainprop} and Theorem \ref{main}, we have.

\begin{thm}\label{th31}Let $k_1$ and $k_2$ be non-negative integers with $k_1\le k_2$ and $k_1+k_2=k\ge3$. Then the following are equivalent:

(a) $M=[-k_1, k_2]^*$ splits $\mathbb{Z}_p$ for some prime $p$ with $p\equiv1\pmod{(k_1+k_2)}$.

(b)There are infinitely
many primes $p$ such  that $[-k_1, k_2]^*$ splits $\mathbb{Z}_p$.

(c) There is  a direct logarithm from $[-k_1, k_2]^*$ to $\mathbb{Z}_k$ for some positive integer $k$ with $(k_1+k_2)|k$.
\end{thm}

As immediate consequences of Theorem \ref{th31}, we have

\begin{coro}Let $k_1$ and $k_2$ be non-negative integers with $k_1\le k_2$ and $k_1+k_2=k\ge3$. Then

(i) If $k_1+k_2+1$ is a prime, then there are infinitely
many primes $p$ such  that $[-k_1, k_2]^*$ splits $\mathbb{Z}_p$.

(ii) If $2k+1$ is a prime, then there are infinitely
many primes $p$ such  that $[1, k]$ splits $\mathbb{Z}_p$.\end{coro}

For the  nonexistence of certain nonsingular splittings, we also need the following result.

\begin{prop}\label{kfac}{\rm (\cite{SzS09}Theorem 7.12)}If $ G = A \cdot B$ is a factorization of the finite abelian group $G$ (written multiplicatively) and $k$ is an integer relatively prime to $|A|$, then $ G = A^k\cdot B$ is a factorization of the abelian group $G$, where $A^k=\{a^k : a \in A\}$.\end{prop}

 We first prove the following more general theorem.

\begin{thm}\label{th32} Let $n=2m$ be an even positive integer. Suppose that $N$ is a subset of the cyclic group $\mathbb{Z}_{2m}$ such that
$\{0, m\}\subseteq N$ and $|N|$ is odd, then $N$ is not a direct factor of $\mathbb{Z}_{2m}$. \end{thm}
\begin{proof} If $N$ is a direct factor of $\mathbb{Z}_{2m}$, then there exists a subset $A$ of $\mathbb{Z}_{2m}$ such that $N+A=\mathbb{Z}_{2m}$ is a factorization. Since $|N|$ is odd, by Proposition \ref{kfac}, $2N+A $ is also a factorization of $\mathbb{Z}_{2m}$, which implies that $|2N|=|N|$. However, $2\cdot0=2\cdot m=0$ in $\mathbb{Z}_{2m}$, it follows that $|2N|\le|N|-1$, a contradiction. \end{proof}

\begin{thm}Let $M$ be a finite set of nonzero integers. Suppose that $\{-1, 1\}\subset M$ and $|M|$ is odd, then $M$ does not split $\mathbb{Z}_p$ for any prime $p$. In particular, if $1\le k_1\le k_2$ and $k_1+k_2$ is odd, then $[-k_1, k_2]^*$ does not split $\mathbb{Z}_p$ for any prime $p$.\end{thm}
\begin{proof}Assume that $M$  splits $\mathbb{Z}_p$ with the splitter set $B$ for some prime $p$. Let $g$ be a primitive root of the prime $p$, $N=\{ind_g (m), m\in M\}$, $S=\{ind_g(m), m\in B\}$, then we have $\{0, (p-1)/2\}\subset N$ and $N+S=\mathbb{Z}_{p-1}$ is a factorization, which contradicts Theorem \ref{th32}. This completes the proof.\end{proof}

\section{Some presentations for the splitter set $B$}

In this section, we will give some presentations for the possible splitter set. We first present the following useful results.

\begin{lem}\label{kongji}Let $M$ be a finite subset of nonzero integers and $1\in M$, and let $X=M/M\backslash\{1\}=\{\frac{m_1}{m_2}: m_1\ne m_2, \, m_1,m_2\in M\}$. Let $p$ be an odd prime with $p\equiv1\pmod{|M|}$. Then we have

(a) If $M$ splits $\mathbb{Z}_p$ with the splitter set $B$, then $B\cap BX=\emptyset$.

(b) If $B$ is a subset of $\mathbb{Z}_p$ such that $1\in B$, $MB=\mathbb{Z}_p\backslash\{0\}$ and $B\cap BX=\emptyset$, then $M$ splits $\mathbb{Z}_p$ with the splitter set $B$.\end{lem}
\begin{proof}(a) On the contrary, there are elements $b_1, b_2\in B$ and $x\in X$ such that $b_1=b_2x$. By the definition of $X$, we have $x=m_1/m_2, m_1, m_2\in M$ and $m_1\ne m_2$. Hence $b_1m_2=b_2m_1$, which contradicts $MB$ being a splitting of $\mathbb{Z}_p$.

(b) By the assumptions, it suffices to prove that for any $ m_1, m_2\in M, b_1, b_2\in B$ with $m_1b_1=m_2b_2$, we have $b_1=b_2$ and $m_1=m_2$. If $m_1b_1=m_2b_2$, then $b_2=b_1\cdot\frac{m_1}{m_2}\in BX$ when $m_1\ne m_2$, so $m_1=m_2$, and hence $b_1=b_2$. It follows that $M$ splits $\mathbb{Z}_p$ with the splitter set $B$, as required. \end{proof}

\begin{lem}(\cite{[PK]}, Theorem IV.1.) \label{only if subgroup}
Let $k_1$, $k_2$ be positive integers with $1\leq k_1\leq k_2$ and let $p$ be an odd prime with $p\equiv 1$(mod $k_1 + k_2)$. Then $M = [-k_1,k_2]^*$ is a direct factor of $\mathbb{Z}_p^*$ if and only if $M$ is a direct factor of the subgroup $H =<-1,2,\cdots ,k_2 >$ of $\mathbb{Z}_p^*$.
\end{lem}

 In \cite{[PK]}, we  obtained a necessary and sufficient conditions for the prime $p$ such that $[-1, 3]^*$ splits $\mathbb{Z}_p$ with the some splitter set $B$. Now we give a presentation of $B$ for $[-1, 5]^*$, we have.
\begin{thm} \label{|T_1/T_2|}
For a prime $p$ with $p\equiv1\pmod{6}$, if $[-1, 5]^*$ splits $\mathbb{Z}_p^*$ with a splitter set $B$ and let $B_1=B\cap <-1, 2, 3, 5>$, then
$$B_1=\bigcup_{k=0}^\infty\varepsilon_k8^k<(-\frac{2}{5}), (-\frac{4}{3})>, \quad \varepsilon_k=1 \, \mbox{or} \, -1,$$ or $$B_1=\bigcup_{k=0}^\infty\varepsilon_k8^k<(-\frac{4}{5}), (-\frac{2}{3})>, \quad \varepsilon_k=1 \, \mbox{or} \, -1.$$
\end{thm}

\begin{proof}
Suppose $[-1, 5]^*$ splits $\mathbb{Z}_p^*$ with a splitter set $B$, then
$$X=\frac{[-1,5]^*}{[-1,5]^*}\setminus \{1\}=\left\{\pm2, \pm3, \pm4, \pm5, \pm\frac{1}{2}, \pm\frac{1}{3}, \pm\frac{1}{4}, \pm\frac{1}{5}, \frac{2}{3}, \frac{3}{2},\frac{2}{5}, \frac{5}{2},\frac{3}{4}, \frac{4}{3}, \frac{3}{5}, \frac{5}{3},  \frac{4}{5}, \frac{5}{4},\right\}.$$
Since $$-2=(-1)\times 2=1\times (-2)=2\times (-1)=3\times (-\frac{2}{3})=4\times (-\frac{1}{2})=5\times (-\frac{2}{5}),$$
by Lemma \ref{kongji} (a) we see that
$$-\frac{2}{3}\in B, \ or \ -\frac{2}{5}\in B. $$
By $$-4=(-1)\times 4=1\times (-4)=2\times (-2)=3\times (-\frac{4}{3})=4\times (-1)=5\times (-\frac{4}{5}),$$
similarly,  by Lemma \ref{kongji} (a) again we get $$-\frac{4}{3}\in B, \ or \ -\frac{4}{5}\in  B. $$
Note that $-\frac{4}{3}=-\frac{2}{3}\times 2$ and $-\frac{4}{5}=-\frac{2}{5}\times 2$, hence we have
$$-\frac{2}{3},-\frac{4}{5}\in B \ or \ -\frac{2}{5},-\frac{4}{3}\in B.$$
For $-2=-\frac{2}{3}\times3 =-\frac{2}{5}\times5$, $-4=-\frac{4}{3}\times3 =-\frac{4}{5}\times 5$, $-\frac{8}{3}=-\frac{2}{3}\times 4=-\frac{4}{3}\times 2$ and $-\frac{8}{5}=-\frac{4}{5}\times 2=-\frac{2}{5}\times 4$, by Lemma \ref{kongji} (a) we have $-4,$ $-2$, $-\frac{8}{3}$, $-\frac{8}{5} \not\in B$.
Since $$-8=(-1)\times 8=1\times (-8)=2\times (-4)=3\times (-\frac{8}{3})=4\times (-2)=5\times (-\frac{8}{5}),$$
so $$-8, \ or \ 8\in  B. $$

Now we first consider the case where $-\frac{2}{5},$ $-\frac{4}{3}\in B$.
We will show that $$<-\frac{2}{5},-\frac{4}{3}>\subseteq B,$$
that is, $(-\frac{2}{5})^a(-\frac{4}{3})^b \in B$ for any non-negative integers $a$ and $b$.

The proof is by induction on $a+b$.
We have proved that $(-\frac{2}{5})^a(-\frac{4}{3})^b \in B$ for $a+b\leq 1.$ (Since $1\in B$.)
Assume that the result hold for all $a+b-1$.
We show that the result then holds for $a+b$.
This can be done by induction on $b$.
If $b=0$, set $M=(-\frac{2}{5})^{a-1}$,
then by induction $M,$ $(-\frac{4}{3})(-\frac{2}{5})^{a-2} \in B$. By Lemma \ref{kongji} (a), we have $-M$,  $2M,$ $-2M,$ $-\frac{1}{2}M$, $-\frac{2}{3}M\times 5=(-\frac{4}{3})(-\frac{2}{5})^{a-2}\times (-1) \not\in B$. Since
 $$-2M=(-1)\times (2M)=1\times (-2M)=2\times (-M)=3\times (-\frac{2}{3}M)=4\times (-\frac{1}{2}M)=5\times (-\frac{2}{5}M),$$
then $-\frac{2}{5}M=(-\frac{2}{5})^a\in B$.
Now suppose that $b>0$ and the result holds  for $b-1$.
Set $M=(-\frac{2}{5})^{a}(-\frac{4}{3})^{b-1}$.
Thus by induction $M,$ $(-\frac{2}{5})^{a+1}(-\frac{4}{3})^{b-1} \in B$. By Lemma \ref{kongji} (a), we have  $4M,$ $-4M,$ $-2M,$ $-M$, $-\frac{4}{5}M=(-\frac{2}{5})^{a+1}(-\frac{4}{3})^{b-1}\times 2\not\in B$.
Since $$-4M=(-1)\times (4M)=1\times (-4M)=2\times (-2M)=3\times (-\frac{4}{3}M)=4\times (-M)=5\times (-\frac{4}{5}M),$$
then $-\frac{4}{3}M=(-\frac{2}{5})^a(-\frac{4}{3})^{b}\in B$ and the proof is complete.

Next, we will use induction on $k$ to prove that
$$-8^k(-\frac{2}{5})^a(-\frac{4}{3})^b\in B \,\, \mbox{or}\,\, 8^k(-\frac{2}{5})^a(-\frac{4}{3})^b\in B$$
for any non-negative integers $a$, $b$ and $k$.

Set $N=8(-\frac{2}{5})^a(-\frac{4}{3})^b$, then we have
$$N=(-1)\times (-N)=1\times N=2\times (\frac{1}{2}N)=3\times (\frac{1}{3}N)=4\times (\frac{1}{4}N)=5\times (\frac{1}{5}N).$$
Observe that $$\frac{1}{2}N\times (-1)=(-\frac{2}{5})^a(-\frac{4}{3})^{b+1}\times 3,$$
$$\frac{1}{3}N\times (-1)=(-\frac{2}{5})^a(-\frac{4}{3})^{b+1}\times 2,$$
$$\frac{1}{4}N=(-\frac{2}{5})^a(-\frac{4}{3})^{b}\times 2$$
and
$$\frac{1}{5}N\times (-1)=(-\frac{2}{5})^{a+1}(-\frac{4}{3})^{b}\times 4.$$
Combining these results with $<-\frac{2}{5},-\frac{4}{3}>\subseteq B$ yields that $-N\in B$ or
$N\in B$. Hence we have proved that it is true for $k=1$. Suppose that $k>1$ and that it is true for $k-1$.
For $8^k(-\frac{2}{5})^a(-\frac{4}{3})^b$, by the same argument as above, we have
$$8^k(-\frac{2}{5})^a(-\frac{4}{3})^b\in B \,\, \mbox{or}\,\,-8^k(-\frac{2}{5})^a(-\frac{4}{3})^b\in B.$$

On the other hand, we observe that if $8^k(-\frac{2}{5})^a(-\frac{4}{3})^b\in B$, then $8^k(-\frac{2}{5})^a(-\frac{4}{3})^b\times(-\frac{2}{5})\not\in B$, $8^k(-\frac{2}{5})^a(-\frac{4}{3})^b\times(-\frac{4}{3})\not\in B$. That is, $-8^k(-\frac{2}{5})^{a+1}(-\frac{4}{3})^b\not\in B$, $-8^k(-\frac{2}{5})^a(-\frac{4}{3})^{b+1}\not\in B$. Hence
$8^k(-\frac{2}{5})^{a+1}(-\frac{4}{3})^b\in B$ and $8^k(-\frac{2}{5})^a(-\frac{4}{3})^{b+1}\in B$. Therefore we have proved that
$$8^k<(-\frac{2}{5}), (-\frac{4}{3})>\subset B.$$

Let
$$B_1=\bigcup_{k=0}^\infty\varepsilon_k8^k<(-\frac{2}{5}), (-\frac{4}{3})>\subseteq B, \quad \varepsilon_k=1 \, \mbox{or} \, -1.$$
For any element $(-1)^{\sigma}2^u3^v5^s\in <-1,2,3,5>$,
$$(-1)^{\sigma}2^u3^v5^s=(-\frac{2}{5})^{-s}(-\frac{4}{3})^{-v}(-1)^{\sigma+s+v}2^{s+2v+u}.$$
Set $s+2v+u=3k+t$ with $t\in \{0,1,2\}$, then $$(-1)^{\sigma}2^u3^v5^s=b\cdot r, \quad b\in B_1, \,\,r\in\{\pm1, \pm2, \pm4\}.$$
If $(-1)^{\sigma}2^u3^v5^s=b\cdot (-2), b\in B_1$, then $(-1)^{\sigma}2^u3^v5^s=b\cdot(-\frac{2}{5})\cdot5$. If  $(-1)^{\sigma}2^u3^v5^s=b\cdot (-4), b\in B_1$, then $(-1)^{\sigma}2^u3^v5^s=b\cdot(-\frac{4}{3})\cdot3$. Hence, we have $[-1, 5]^*B_1=<-1,2,3,5>$. By Lemma \ref{only if subgroup}, $B$ exists if and only if $[-1, 5]^*B_1=<-1,2,3,5>$ is a direct sum. Hence $B_1=\bigcup_{k=0}^\infty\varepsilon_k8^k<(-\frac{2}{5}), (-\frac{4}{3})>\subseteq B, \quad \varepsilon_k=1 \, \mbox{or} \, -1$.

If $-\frac{2}{3},-\frac{4}{5}\in B$, we obtain a similar result by the same argument. This completes the proof.

\end{proof}

{\bf Remark:} We can also obtain some similar results for other $M=[-k_1, k_2]^*$. 

For $B_1=\bigcup_{k=0}^\infty\varepsilon_k8^k<(-\frac{2}{5}), (-\frac{4}{3})>\subseteq B, \quad \varepsilon_k=1 \, \mbox{or} \, -1$, we take $k=6$, $r=4$, $a_1=-1$,  $a_2=2$, $a_3=3$, $a_4=5$, $\varepsilon_1=-1$, $\varepsilon_2=e^{i\pi/3}$, $\varepsilon_3=-e^{2i\pi/3}$, $\varepsilon_4=-e^{i\pi/3}$. By the Lemma in Mills \cite{[M]}, we have $a_1^{v_1}a_2^{v_2}a_3^{v_3}a_4^{v_4}=\beta^6, \beta\in\mathbb{Q}(e^{i\pi/3}), 1\le v_i\le6, i=1,2, 3, 4$ if and only if $a_1^{v_1}a_3^{v_3}=-3^3$ and $v_2=v_4=6$, hence $N(6, 4)=3(-(-e^{2i\pi/3})^3+1)=6>0$, where $N(6, 4)=\sum_{v_1=1}^6\sum_{v_1=2}^6\sum_{v_3=1}^6\sum_{v_4=1}^6(\varepsilon_1^{v_1}\varepsilon_2^{v_2}\varepsilon_3^{v_3}\varepsilon_4^{v_4})$, and the summation is over all $(v_1, v_2, v_3, v_4)$ such that $a_1^{v_1}a_2^{v_2}a_3^{v_3}a_4^{v_4}=\beta^6, \beta\in\mathbb{Q}(e^{i\pi/3}$. Therefore by Theorem 1 in Elliott \cite{E70}, there are infinitely many prime $p$ such that
$$\left(\frac{a_i}{p}\right)_6\equiv\varepsilon_i\pmod{p}, \quad i=1,2, 3, 4.$$
Finally, we prove that for the above prime $p$, $[-1, 5]^*$ splits $\mathbb{Z}_p$. To do this, we first prove that $[-1, 5]^*B'=<-1, 2, 3, 5>$, where $B'=\bigcup_{k=0}^\infty(-8)^k<(-\frac{2}{5}), (-\frac{4}{3})>$ is a factorization. If $m_1b_1=m_2b_2, m_1, m_2\in[-1, 5]^*, b_1, b_2\in B'$, then we have
$$\left(\frac{m_1b_1}{p}\right)_6=\left(\frac{m_2b_2}{p}\right)_6\pmod{p}.$$
A simple computation shows that $\left(\frac{m_1}{p}\right)_6=\left(\frac{m_2}{p}\right)_6$, which yields $m_1=m_2$, and thus $b_1=b_2$. This proves that  $[-1, 5]^*B'=<-1, 2, 3, 5>$ is a factorization, so  $[-1, 5]^*B_1=<-1, 2, 3, 5>$ is also a factorization. Therefore $[-1, 5]^*$ splits $\mathbb{Z}_p$ by Lemma \ref{only if subgroup}.

For $B_1=\bigcup_{k=0}^\infty\varepsilon_k8^k<(-\frac{4}{5}), (-\frac{2}{3})>\subseteq B, \quad \varepsilon_k=1 \, \mbox{or} \, -1$, we take $k=6$, $r=4$, $a_1=-1$,  $a_2=2$, $a_3=3$, $a_4=5$, $\varepsilon_1=-1$, $\varepsilon_2=e^{2i\pi/3}$, $\varepsilon_3=-e^{2i\pi/3}$, $\varepsilon_4=e^{i\pi/3}$. We obtain that $[-1, 5]^*B_1=<-1, 2, 3, 5>$ is  a factorization, and hence $[-1, 5]^*$ splits $\mathbb{Z}_p$.

{\bf Example 4.1:}  For the set $[-4, 4]^*$,  we take $k=16$, $r=3$, $a_1=-1$,  $a_2=2$, $a_3=3$, $\varepsilon_1=-1$, $\varepsilon_2=e^{i\pi/4}$, $\varepsilon_3=e^{3i\pi/4}$. By the Lemma in Mills \cite{[M]}, we have $a_1^{v_1}a_2^{v_2}a_3^{v_3}=\beta^{16}, \beta\in\mathbb{Q}(e^{i\pi/8}), 1\le v_i\le8, i=1,2, 3$ if and only if $a_1^{v_1}=1$, $v_2=8$ or $16$ and $v_3=16$, hence $N(16, 3)=8((e^{i\pi/4})^8+1)=16>0$. Therefore by Theorem 1 in Elliot \cite{E70}, there are infinitely many prime $p$ such that
$$\left(\frac{a_i}{p}\right)_{16}\equiv\varepsilon_i\pmod{p}, \quad i=1,2, 3.$$
Let $B$ be the kernel of the homomorphism $\varphi: \mathbb{Z}_p^*\to\mathbb{Z}_p^*$ given by $\varphi(a)\equiv a^{p-1/16}$, it is easy to check that $[-4, 4]^*B=\mathbb{Z}_p^*$ is a factorization, so $[-4, 4]^*$ splits $\mathbb{Z}_p$ with splitter set $B$.
\section*{Acknowledgments}
The author thanks Professor Qing Xiang
for useful comments and suggestions.


\begin{thebibliography}{99}








\bibitem{BM12} S. R. Blackburn and J. F. McKee, Constructing $k$-radius sequences, Math. Comput. 81(2012) 2439-2459.

\bibitem{BE13} S. Buzaglo and T. Etzion, Tilings with $n$-dimensional chairs and their
applications to asymmetric codes, IEEE Trans. Inf. Theory 59(2013) 1573-1582.

\bibitem{CSBB10} Y. Cassuto, M. Schwartz, V. Bohossian and J. Bruck, Codes for
asymmetric limited-magnitude errors with application to multilevel flash
memories, IEEE Trans. Inf. Theory 56(2010) 1582-1595.


\bibitem{EB10} N. Elarief and B. Bose, Optimal, systematic, $q$-ary codes correcting
all asymmetric and symmetric errors of limited magnitude, IEEE Trans. Inf. Theory 56(2010) 979-983.

\bibitem{E70} P.D.T.A. Elliott, The distribution of power residues and certain related results, Acta Arith.
17 (1970), 141-159.

\bibitem{GS81} S. Galovich and S. Stein, Splittings ofabelian groups by integers,  Aequationes Math. 22(1981) 249-267.

\bibitem{H73} W. Hamaker, Factoring group and tiling space. Aequationes Math. 9 (1973), 145-149.

\bibitem{HS74} W. Hamaker and  S. Stein, Splitting groups by integers, Proc. Amer. Math. Soc. 46(1974) 322-324.



\bibitem{H83} D. Hickerson, Splittings of finite groups, Pacific J. Math. 107(1983) 141-171.

\bibitem{HS86} D. Hickerson and S. Stein, Abelian groups and packing by semicrosses, Pacific J. Math. 122(1986) 95-109.

\bibitem{HP03} W. C. Huffman and V. Pless,  Fundamentals of error-correcting codes. Cambridge University Press, Cambridge, 2003.


\bibitem{KBE11} T. Kl${\o}$ve, B. Bose and N. Elarief, Systematic, single limited magnitude
error correcting codes for flash memories, IEEE Trans. Inf. Theory 57(2011) 4477-4487.

\bibitem{KLNY11} T. Kl${\o}$ve, J. Luo, I. Naydenova and S. Yari, Some codes correcting
asymmetric errors of limited magnitude, IEEE Trans. Inf. Theory 57(2011) 7459-7472.


\bibitem{KLY12} T. Kl${\o}$ve, J. Luo and S. Yari, Codes correcting single errors of limited
magnitude, IEEE Trans. Inf. Theory 58(2012) 2206-2219.

\bibitem{M96}  S. Martirosyan, Single-error correcting close packed and perfect codes,
in Proc. 1st INTAS Int. Seminar Coding Theory Combinat. (1996) 90-115.

\bibitem{[M]}  W.H. Mills, Characters with preassigned values, Canad. J. Math. 15 (1963), 169-171.


\bibitem{MMSJ07} K. Momihara, M. M$\ddot{u}$ller, J. Satoh and M. Jimbo, Constant weight
conflict-avoiding codes, SIAM J. Discrete Math. 21(2007) 959-979.



\bibitem{OSW18} O. Roche-Newton, I. D. Shkredov and A. Winterhof, Packing sets over finite abelian groups. Integers  18(2018) Paper  A38, 9 pp.


\bibitem{Sa57} A. D. Sands,  On the factorisation of finite abelian groups,  Acta Math. Acad. Sci. Hungar. 8(1957) 65-86.
\bibitem{Sc12} M. Schwartz, Quasi-cross lattice tilings with applications to flash
memory, IEEE Trans. Inf. Theory 58(2012) 2397-2405.
\bibitem{Sc14} M. Schwartz, On the non-existence of lattice tilings by quasi-crosses,
Eur. J. Combinat. 36(2014) 130-142.

\bibitem{SWC10} K. W. Shum, W. S. Wong and C. S. Chen, A general upper bound on
the size of constant-weight conflict-avoiding codes,  IEEE Trans. Inf.
Theory 56(2010) 3265-3276.

\bibitem{St67} S. K. Stein, Factoring by subsets, Pacific J. Math. 22(1967) 523-541.

\bibitem{St84} S. K. Stein, Packings of $R^n$ by certain error spheres, IEEE Trans. Inf. Theory 30(1984) 356-363.

\bibitem{SS94} S. Stein and S. Szab$\acute{o}$, Algebra and Tiling (Carus Mathematical
Monographs), vol. 25. Washington, DC, USA: MAA,
1994.

\bibitem{Su05} Z. Sun, Quartic residues and binary quadratic forms, J. Number Theory 113(2005) 10-52.

\bibitem{Sz86} S. Szab$\acute{o}$, Some problems on splittings of groups,  Aequationes Math. 30(1986) 70-79.

\bibitem{[SS]} Szab\'{o}, S\'{a}ndor, Some problems on splittings of groups. II, {\it Proceedings of the American Mathematical Society.}, 101.4(1987),585-585.

\bibitem{Sz87} S. Szab$\acute{o}$, Some problems on splittings of groups II,  Proc. Amer. Math.
Soc. 101(1987) 585-591.
\bibitem{SzS09} S. Szab$\acute{o}$ and A. D. Sands, Factoring Groups into Subsets (Lecture Notes
in Pure and Applied Mathematics), vol. 257. Boca Raton, FL, USA: CRC Press, 2009.

 \bibitem{W95} A. J. Woldar, A reduction theorem on purely singular splittings of cyclic
groups, Proc. Amer. Math. Soc. 123(1995) 2955-2959.

\bibitem{YKB13} S. Yari, T. Kl${\o}$ve and B. Bose, Some codes correcting
unbalanced errors of limited magnitude for flash memories, IEEE Trans. Inf. Theory 59(2013) 7278-7287.

\bibitem{SS09}S. Szab$\acute{o}$ and A. D. Sands, Factoring Groups Into Subsets (Lecture Notes
in Pure and Applied Mathematics). Boca Raton, FL, USA: CRC Press,
2009.




\bibitem{PK20} P. Yuan, K. Zhao,  On the Existence of Perfect Splitter Sets. Finite Fields Appl.  61(2020), 101603.

\bibitem{[PK]} P. Yuan, K. Zhao, {\it On the Existence of Perfect Splitter Sets.}, submitted (2019).




\end{thebibliography}

\end{document}